\documentclass{amsart}
\usepackage[utf8]{inputenc}
\usepackage{amsmath}
\usepackage{amssymb}
\usepackage{amsthm}
\usepackage{enumitem}
\usepackage{todonotes}
\usepackage{hyperref}
\usepackage{url}
\usepackage{tikz}
\usetikzlibrary{arrows.meta, arrows}

\setlist[enumerate]{label=\rm{(\arabic*)}, ref=(\arabic*)}

\newcommand{\NN}{\mathbb{N}}

\newcommand{\concat}{%
  \mathord{
    \mathchoice
    {\raisebox{1ex}{\scalebox{.7}{$\frown$}}}
    {\raisebox{1ex}{\scalebox{.7}{$\frown$}}}
    {\raisebox{.7ex}{\scalebox{.5}{$\frown$}}}
    {\raisebox{.7ex}{\scalebox{.5}{$\frown$}}}
  }
}

\newtheorem{thm}{Theorem}
\newtheorem*{thm*}{Theorem}
\newtheorem{lemma}[thm]{Lemma}
\newtheorem{prop}[thm]{Proposition}

\theoremstyle{definition}

\newtheorem{defn}[thm]{Definition}
\newtheorem{rem}[thm]{Remark}

\newtheorem{example}[thm]{Example}

\newenvironment{fsa}[1][auto]{
    \begin{tikzpicture}[#1,->,>=Stealth,shorten >=1pt,auto,
      node distance=2.8cm,semithick,double equal sign distance]
      \tikzstyle{state}=[fill=none,draw=black,text=black, shape=circle]}{%
    \end{tikzpicture}}

\def\9{{\mathbf 0}}
\def\8{{\mathbf 1}}
\def\7{{\mathbf 2}}

\title{Infinite finitely generated automata semigroups have infinite orbits}

\author{Dominik Francoeur}

\date{\today}

\begin{document}

\begin{abstract}
We prove that the semigroup generated by a finite state Mealy automaton $\mathcal{A}=(Q,A,\tau)$ is infinite if and only if there exists some right-infinite word in the alphabet $A$ with infinite orbit.
\end{abstract}

\maketitle

\section{Introduction}

A Mealy automaton is a machine that takes as input a letter from an alphabet $A$ and, depending on its currently active state, outputs a letter and modifies its active state (see Figure \ref{figure:MooreDiagram} for an example).
\begin{figure}[h]
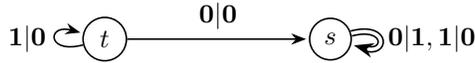

\centering
\[\begin{fsa}[baseline]
    \node[state] (t) at (0,0) {$t$};
    \node[state] (s) at (3,0) {$s$};
    \path (t) edge node {$\9|\9$} (s)
    	  (t) edge [loop left] node {$\8|\9$} (t)
          (s) edge [double,loop right] node {$\9|\8,\8|\9$} (s);
  \end{fsa}
\]
\caption{The Moore diagram of a Mealy automaton.}\label{figure:MooreDiagram}
\end{figure}

Such a machine defines a semigroup acting on the Cantor set $A^\omega$. The generators correspond to the states of the automaton ($\{s,t\}$ in the example) and the image of $a_0a_1a_2\dots\in A^\omega$ by the generator $q$ is the unique right-infinite word $b_0b_1b_2\dots\in A^\omega$ such that $(a_0|b_0)(a_1|b_1)\dots$ is a path starting at $q$.

Despite being produced by very simple rules, semigroups generated by Mealy automata can be quite rich and have some striking properties. For instance, one can find among this class finitely generated infinite torsion groups \cite{Grigorchuk80} as well as groups \cite{Grigorchuk83} and semigroups \cite{BartholdiReznykovSuschansky06} of intermediate growth.

One can wonder what sort of restrictions being recursively defined from a finite amount of data imposes on the algorithmic, structural or dynamical properties of semigroups generated by finite state Mealy automata. D'Angeli, Rodero and Wächter recently asked (\cite{AngeliRoderoWachter17}, Open Problem 3) whether such a semigroup could be infinite while every orbit of its canonical action on the Cantor set is finite. In this note, we prove that this is not possible (Theorem \ref{thm:ExistInfiniteOrbit}):

\begin{thm*}
Let $S$ be the semigroup of endomorphisms of the Cantor set $A^\omega$ generated by a finite state Mealy automaton $\mathcal{A}=(Q,A,\tau)$. Then $S$ is infinite if and only if there exists $\xi \in A^\omega$ such that $|S\cdot \xi| = \infty$.
\end{thm*}
In fact, our proof yields the slightly stronger following result (Theorem \ref{thm:WorksForSubgroups}):
\begin{thm*}
Let $S$ be as above. Then a subsemigroup $T\leq S$ is infinite if and only if there exists $\xi \in A^\omega$ such that $|T\cdot \xi| = \infty$.
\end{thm*}

In Section \ref{sec:Preliminaries}, we briefly define the relevant notions and set the notation. The theorem is proved in Section \ref{sec:Proof}. Section \ref{sec:Remarks} contains remarks and examples related to possible generalisations of this result.

\subsection*{Acknowledgement} The author would like to thank Laurent Bartholdi and Ivan Mitrofanov for many useful remarks and discussions. This work was supported by a Doc.Mobility grant from the Swiss National Science Foundation as well as the "@raction" grant ANR-14-ACHN-0018-01 while the author was visiting the École Normale Supérieure in Paris.

\section{Preliminaries}\label{sec:Preliminaries}

\subsection{Alphabets and words}

Let $A$ be a finite set and $A^*$ be the set of words in the alphabet $A$. The concatenation between two words $u,v\in A^*$ will be denoted by $u\concat v$. There is a natural prefix order on $A^*$ given by $u\leq v$ if $u$ is a \emph{prefix} of $v$, that is, if there exists a (possibly empty) word $w\in A^*$ such that $v = u\concat w$.

We will denote by $A^\omega$ the set of right-infinite sequences of letters in the alphabet $A$. The concatenation of a finite word $u\in A^*$ on the left and a right-infinite sequence $\xi \in A^\omega$ on the right will be written as $u\concat \xi$. Given $u\in A^*$ and $\xi \in A^\omega$, we will write $u\leq \xi$ if $u$ is a prefix of $\xi$ (i.e. if there exists $\zeta \in A^\omega$ such that $\xi = u\concat \zeta$). Given a strictly increasing sequence $(u_i)_{i\in\NN}$ of elements of $A^*$, there exists a unique $\xi\in A^\omega$ such that $u_i \leq \xi$ for all $i\in \NN$.

\subsection{Mealy automata}

A \emph{Mealy automaton} is a tuple $\mathcal{A} = (Q,A,\tau)$, where $Q$ is a set called the \emph{set of states}, $A$ is a finite alphabet and $\tau$ is a map
\[
\tau \colon \begin{cases}
Q\times A &\rightarrow A \times Q \\
(q,a) &\mapsto (q\cdot a, q@a).
\end{cases}
\]
If $Q$ is finite, $\mathcal{A}$ is called a \emph{finite state} Mealy automaton.

A Mealy automaton can be represented by its \emph{Moore diagram}, which is a directed labelled graph with set of vertices $Q$. For each $q\in Q$ and $a\in A$, there is an edge from $q$ to $q@a$ labelled by $a|(q\cdot a)$. See Figure \ref{figure:MooreDiagram} for an example.

\subsection{Semigroup generated by a Mealy automaton}

Let $\mathcal{A} = (Q,A, \tau)$ be a Mealy automaton. For all $q\in Q$, we have a map $q\cdot \colon A \rightarrow A$ that sends $a$ to $q\cdot a$. We can inductively extend this map to a map $q\cdot \colon A^* \rightarrow A^*$ by
\[q\cdot (a\concat u) = (q\cdot a)\concat ((q@a)\cdot u)\]
for $a\in A$ and $u\in A^*$. Thus, $Q$ can be seen as a set of endomorphisms of $A^*$. We will denote by $S = \langle Q \rangle_+ \leq \text{End}(A^*)$ the semigroup generated by $Q$.

For all $u\in A^*$, we can inductively define a map $@u \colon Q \rightarrow Q$ by the formula
\[q@u = q@(a\concat v) = (q@a)@v\]
where $a\in A$, $v\in A^*$ are such that $u=a\concat v$. This map extends to a map $@u\colon S\rightarrow S$ by the formula
\[(st)@u = (s@(t\cdot u))(t@u).\]
We then see that for all $u,v\in A^*$ and all $s\in S$, we have
\[s\cdot (u\concat v) = (s\cdot u)\concat((s@u)\cdot v).\]

Given $s\in S$ and a strictly increasing sequence $(u_i)_{i\in \NN}$ of elements of $A^*$, it follows from the above formula that the sequence $(s\cdot u_i)_{i\in\NN}$ is also strictly increasing. Thus, the action of $S$ on $A^*$ extends to an action of $S$ on $A^\omega$. We then have
\[s\cdot (u\concat \xi) = (s\cdot u)\concat ((s@u)\cdot \xi)\]
for all $s\in S$, $u\in A^*$ and $\xi \in A^\omega$.

It is easy to construct a Mealy automaton such that the semigroup $S$ it generates is infinite, but for every $\xi \in A^\omega$, the orbit $S\cdot \xi$ of $\xi$ under the action of $S$ is a finite set (see for example Figure \ref{figure:FiniteOrbits}).

\begin{figure}
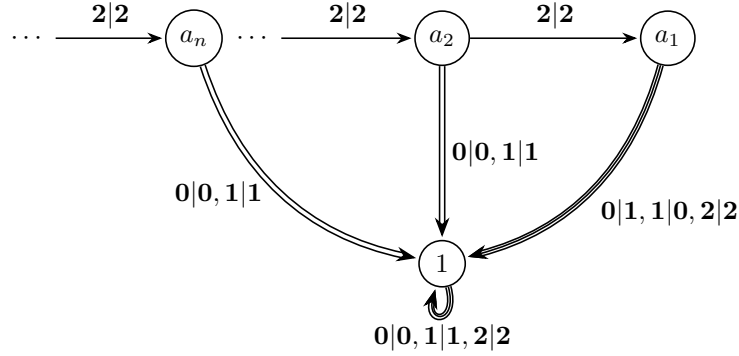

\centering

\[\begin{fsa}[baseline]
    \node[state] (a) at (3,0) {$a_1$};
    \node[state] (b) at (0,0) {$a_2$};
    \node[state] (c) at (-3.3,0) {$a_n$};
    \node (l1) at (-2.5,0) {$\cdots$};
    \node (l2) at (-5.5, 0) {$\cdots$};
    \node[state] (e) at (0,-3) {$1$};
    \path (a) edge [double, bend left] node {$\9|\8, \8|\9, \7|\7$} (e)
    	  (a) edge [bend left] (e)
    	  (b) edge [double] node {$\9|\9, \8|\8$} (e)
    	  (b) edge node {$\7|\7$} (a)
    	  (l1) edge node {$\7|\7$} (b)
    	  (l2) edge node {$\7|\7$} (c)
    	  (c) edge [double, bend right] node[below, left, xshift=-4] {$\9|\9, \8|\8$} (e)
          (e) edge [double,loop below] node {$\9|\9,\8|\8, \7|\7$} (e)
          (e) edge [loop below] (e);
  \end{fsa}
\]
\caption{A Mealy automaton generating an infinite semigroup for which every orbit is finite.}\label{figure:FiniteOrbits}
\end{figure}

However, as we will prove in the next section, in the case of a finite state Mealy automaton, there must exist an infinite orbit as soon as the semigroup it generates is infinite.

\section{Proof of the theorem}\label{sec:Proof}

As in the previous section, let $\mathcal{A} = (Q,A,\tau)$ be a Mealy automaton and $S$ be the semigroup it generates. In what follows, we will assume that $Q$ is finite.

\begin{defn}
For $v\in A^*$, we define $m_v\in \NN\cup \{\infty\}$ by
\[m_v = \sup_{\xi\in A^{\omega}}|S\cdot (v\concat \xi)|.\]
\end{defn}
\begin{rem}\label{remark:MvOfChildIsAlsoInfinite}
For any $v\in A^*$, we have $m_v = \max_{a\in A}m_{v\concat a}$. Therefore, if $m_v=\infty$, there exists $a\in A$ such that $m_{v\concat a} = \infty$.
\end{rem}

\begin{lemma}\label{lemma:ExistInfiniteMv}
If $S$ is infinite, $m_{\epsilon}=\infty$, where $\epsilon\in A^*$ is the empty word.
\end{lemma}
\begin{proof}
If $m_\epsilon < \infty$, then it follows from Remark \ref{remark:MvOfChildIsAlsoInfinite} that $m_v<\infty$ for all $v\in A^*$. Therefore, there must exist some $M\in \NN$ such that $m_v\leq M$ for all $v\in A^*$. Indeed, assume that this is not the case. We can then find a strictly increasing sequence $(m_{v_n})_{n\in \NN}$. As Remark \ref{remark:MvOfChildIsAlsoInfinite} implies that $m_\epsilon \geq m_{v_n}$ for all $n\in \NN$, we get $m_\epsilon = \infty$, a contradiction.

Hence, if $m_\epsilon<\infty$, then there exists some $M\in \NN$ such that $m_v \leq M$ for all $v\in A^*$. This means that the size of the orbits of the action of the semigroup $S$ on $A^\omega$ is bounded by $M$. Since $S$ is finitely generated, there are only finitely many different homomorphisms from $S$ to the endomorphisms of a set of size at most $M$. This, combined with the fact that the action of $S$ on $A^*$ is faithful, implies that $S$ is finite (a detailed proof of this fact is given in \cite{AngeliRoderoWachter17}, Proposition 3).
\end{proof}

For $v\in A^*$, let $\pi_v \colon S \rightarrow \text{End}(S\cdot v)$ be the homomorphism of semigroups given by $\pi_v(s)(w) = s\cdot w$ for all $w\in S\cdot v$. To shorten the notation, when there is no ambiguity, we will write $s\cdot $ for $\pi_v(s)$.

The set $S^{S\cdot v}$ of maps from $S\cdot v$ to $S$ becomes a semigroup when endowed with the pointwise product. There is a natural right action of $\text{End}(S\cdot v)$ on $S^{S\cdot v}$ given by precomposition. We can thus consider the semigroup $\text{End}(S\cdot v) \ltimes S^{S\cdot v}$.

\begin{prop}
For all $v\in A^*$, the map
\[
\psi_v\colon
\begin{cases}
S &\rightarrow \text{End}(S\cdot v) \ltimes S^{S\cdot v} \\
s&\mapsto (s\cdot, s@)
\end{cases}
\]
is a homomorphism (where $s@ \in S^{S\cdot v}$ is the map that sends $w$ to $s@w$).
\end{prop}
\begin{proof}
It is clear that $\psi_v$ maps the identity to the identity. Consider $s_1,s_2 \in S$. We have
\begin{align*}
\psi_v(s_1)\psi_v(s_2) &= (s_1\cdot, s_1@)(s_2\cdot, s_2@)\\
&= ((s_1s_2)\cdot, (s_1s_2)@)\\
&= \psi_v(s_1s_2).
\end{align*}
Indeed, for any $w\in S\cdot v$, we have
\begin{align*}
(s_1s_2)@w &= (s_1@(s_2\cdot w))(s_2@w) \\
&=(s_1@(s_2\cdot))(w)( s_2@)(w) \\
&=((s_1@(s_2\cdot)) (s_2@)) (w). \qedhere
\end{align*}
\end{proof}

\begin{rem}\label{remark:GeneratorsToGenerators}
For every $v\in A^*$, we have $\psi_v(Q) \subseteq \text{End}(S\cdot v) \times Q^{S\cdot v}$, a finite set.
\end{rem}

\begin{rem}\label{remark:HomomorphismThroughBijection}
For $v\in A^*$, if there exists a bijection $\varphi\colon S\cdot v \rightarrow B$ for some set $B$, then the map
\begin{align*}
\psi_v^{\varphi} \colon S &\rightarrow \text{End}(B)\ltimes S^B \\
s&\mapsto (\varphi \circ \pi_v(s)\circ \varphi^{-1}, s@\varphi^{-1})
\end{align*}
is a homomorphism.
\end{rem}

\begin{lemma}\label{lemma:SameHomomorphismImpliesSameMv}
Let $v_1, v_2 \in A^*$ be two words in the alphabet $A$. If there exists a bijection $\varphi \colon S\cdot v_1 \rightarrow S\cdot v_2$ such that $\varphi(v_1) = v_2$ and $\psi_{v_1}^{\varphi} = \psi_{v_2}$, then $m_{v_1\concat w} = m_{v_2\concat w}$ for all $w\in A^*$.
\end{lemma}
\begin{proof}
We will prove that for all $\xi \in A^\omega$, we have $|S\cdot (v_1\concat\xi)| = |S\cdot (v_2\concat\xi)|$. The result is then immediate.

Notice that if there exist $s_1,s_2 \in S$ such that $s_1\cdot (v_1\concat\xi) = s_2\cdot (v_1\concat\xi)$, then $s_1\cdot (v_2\concat\xi) = s_2\cdot (v_2\concat\xi)$. Indeed, 
\[s_i\cdot (v_1\concat\xi) = (s_i\cdot v_1)\concat((s_i@v_1) \cdot \xi).\]
Hence, $s_1\cdot v_1 = s_2 \cdot v_1$ and $(s_1@v_1) \cdot \xi = (s_2@v_1) \cdot \xi$.

By the hypothesis, we have $\varphi\circ \pi_{v_1}(s_i)\circ \varphi^{-1} = \pi_{v_2}(s_i)$. Hence,
\begin{align*}
s_1\cdot v_2 &= \pi_{v_2}(s_1)(v_2)\\
&= \varphi \circ \pi_{v_1}(s_1) \circ \varphi^{-1}(v_2) \\
&= \varphi( \pi_{v_1}(s_1) (v_1)) \\
&= \varphi (\pi_{v_1}(s_2) (v_1)) \\
&= \varphi \circ \pi_{v_1}(s_2) \circ \varphi^{-1} \\
&=s_2 \cdot v_2.
\end{align*}
Furthermore, since $s_i@\varphi^{-1}(w) = s_i@w$ for all $w\in S\cdot v_2$, we get
\[s_i@v_2 = s_i@\varphi^{-1}(v_1) = s_i@v_1.\]
Therefore, $(s_1@v_2)\cdot \xi = (s_2@v_2)\cdot \xi$. It follows that we have $s_1\cdot (v_2\concat\xi) = s_2\cdot (v_2\concat\xi)$.

We can thus define the map
\[
\lambda \colon
\begin{cases}
S\cdot (v_1\concat\xi) &\rightarrow S\cdot (v_2\concat\xi) \\
s\cdot (v_1\concat\xi) &\mapsto s\cdot (v_2\concat\xi).
\end{cases}
\]
It is clearly surjective, and by the symmetry of the previous argument, it must also be injective. This concludes the proof.
\end{proof}

\begin{thm}\label{thm:ExistInfiniteOrbit}
The semigroup $S$ is infinite if and only if there exists $\xi\in A^\omega$ such that $|S\cdot\xi| = \infty$.
\end{thm}
\begin{proof}
It is clear that if there exists some $\xi\in A^\omega$ such that $|S\cdot \xi| = \infty$, then $S$ must be infinite. Let us prove the converse.

Suppose that $S$ is infinite. For each $u\in A^*$, we define the set
\[T_u = \left\{v\in A^* \mid v\geq u \text{ and } m_v=\infty\right\}.\]
It follows from Remark \ref{remark:MvOfChildIsAlsoInfinite} that $T_u\ne \emptyset$ if and only if $m_u=\infty$, and we know from Lemma \ref{lemma:ExistInfiniteMv} that $T_\epsilon \ne \emptyset$.

We may assume that there exists some $w\in A^*$ with $m_w = \infty$ and such that for all $v\in T_w$,
\[|S\cdot w| = |S\cdot v|.\]
Indeed, suppose that this is not the case. Then, for every $u\in A^*$ with $m_u=\infty$, there exists some $v\in T_u$ such that $|S\cdot u| < |S\cdot v|$. Thus, starting from $\epsilon$, which satisfies $m_\epsilon = \infty$, we can construct a strictly increasing sequence $(u_i)_{i\in\NN}$ of elements of $A^*$ such that $u_{i+1}\in T_{u_i}$ and $|S\cdot u_i| < |S\cdot u_{i+1}|$ for all $i\in\NN$. Since the sequence $(u_i)_{i\in \NN}$ is strictly increasing, there exists a unique element $\xi \in A^{\omega}$ such that $u_i$ is a prefix of $\xi$ for all $i\in \NN$. It follows from the fact that $|S\cdot u_i|<|S\cdot u_{i+1}|$ that the size of the orbit of $\xi$ is unbounded, which is what we were looking for. Thus, it suffices to consider the case where there exists some $w\in A^*$ as above.

In order to shorten the notation, let us write $B = S\cdot w$. For all $v\in T_w$, the map
\[
\varphi_v \colon
\begin{cases}
S\cdot v &\rightarrow B \\
s\cdot v &\mapsto s\cdot w
\end{cases}
\]
is a well-defined bijection between $S\cdot v$ and $B$ that satisfies $\varphi_v(v) = w$. The fact that it is well-defined  and surjective follows immediately from the fact that $v\geq w$, and it then must be a bijection, since $|S\cdot v| = |B|$.

We know from Lemma \ref{lemma:SameHomomorphismImpliesSameMv} that for all $a\in A$ and $v\in T_w$, the value of $m_{v\concat a}$ depends only on the homomorphism $\psi_v^{\varphi_v} \colon S \rightarrow \text{End}(B)\ltimes S^B$. However, there are only finitely many such homomorphisms. Indeed, a homomorphism from $S$ to $\text{End}(B)\ltimes S^B$ is uniquely determined by the image of the generating set $Q$, and since by Remark \ref{remark:GeneratorsToGenerators} we have $\psi_v^{\varphi_v}(Q) \subseteq \text{End}(B)\times Q^B$, we have only finitely many choices, hence finitely many such homomorphisms. Consequently, the set
\[K=\left\{m_{v\concat a}\in \NN \cup \{\infty\} \mid v\in T_w, a\in A\right\}\]
is finite. If $K\ne \{\infty\}$, set $M = \max \left(K \setminus \{\infty\}\right)$. Otherwise, set $M=|B|$. Notice that we always have $M\geq |B|$.

Let us consider $x\in w\concat A^*$. If $x\in T_w$, we have $|S\cdot x| = |B| \leq M$. If $x\notin T_w$, then it follows from Remark \ref{remark:MvOfChildIsAlsoInfinite} that there exists $v\in T_w$ and $a\in A$ such that $m_x \leq m_{v\concat a} < \infty$. Then, by the definition of $M$, we get $m_x\leq M$, which implies that $|S\cdot x| \leq M$. Thus, for all $x\in w\concat A^*$, we have $|S\cdot x| \leq M$. This implies that $m_w \leq M$, a contradiction.

We conclude that there must exist some $\xi \in A^\omega$ with infinite orbit.
\end{proof}

\section{Remarks and examples}\label{sec:Remarks}

It was pointed out to us by Ivan Mitrofanov that the proof of Theorem \ref{thm:ExistInfiniteOrbit} also works, with minor modifications, if we consider not only $S$ but any finitely generated subsemigroup of $S$. Thus, we get the following result:

\begin{thm}\label{thm:WorksForSubgroups}
Let $S$ be the semigroup generated by a finite state automaton $\mathcal{A}=(Q,A,\tau)$ and $T\leq S$ be a finitely generated subsemigroup of $S$. Then $T$ is infinite if and only if there exists $\xi\in A^\omega$ such that $|T\cdot \xi| = \infty$.
\end{thm}
\begin{proof}
The proof of Theorem \ref{thm:ExistInfiniteOrbit} yields the result if one replaces $S$ by $T$ and the maps $\psi_v\colon S \rightarrow \text{End}(S\cdot v)\ltimes S^{S\cdot v}$ by maps $\psi_v\colon T \rightarrow \text{End}(T\cdot v)\ltimes S^{T\cdot v}$. As $T$ is finitely generated, there are only finitely many of these maps (up to conjugation by bijections), so the argument goes through.
\end{proof}

Notice that in Theorem \ref{thm:WorksForSubgroups}, for the conclusion to hold, the number of states of the automaton has to be finite. If we allow an infinite number of states, then one can obtain finitely generated subgroups where every orbit is finite. For example, for the automaton described in Figure \ref{figure:SubsemigroupExample}, the subsemigroup generated by $x_0$ is infinite, but each of its orbits are finite.

\begin{figure}[h]
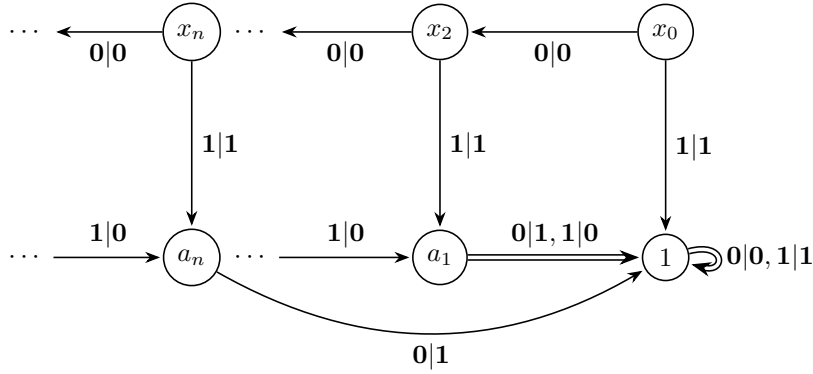

\centering

\[\begin{fsa}[baseline]
    \node[state] (x0) at (3,0) {$x_0$};
    \node[state] (x1) at (0,0) {$x_2$};
    \node[state] (xn) at (-3.3,0) {$x_n$};
    \node[state] (a1) at (0,-3) {$a_1$};
    \node[state] (an) at (-3.3,-3) {$a_n$};
    \node (l1) at (-2.5,0) {$\cdots$};
    \node (l2) at (-5.5, 0) {$\cdots$};
    \node (l3) at (-2.5,-3) {$\cdots$};
    \node (l4) at (-5.5, -3) {$\cdots$};
    \node[state] (e) at (3,-3) {$1$};
    \path (x0) edge node {$\8|\8$} (e)
    	  (x0) edge node {$\9|\9$} (x1)
    	  (x1) edge node {$\9|\9$} (l1)
    	  (x1) edge node {$\8|\8$} (a1)
    	  (xn) edge node {$\9|\9$} (l2)
    	  (a1) edge [double] node {$\9|\8, \8|\9$} (e)
    	  (l3) edge node {$\8|\9$} (a1)
    	  (an) edge [bend right] node [below] {$\9|\8$} (e)
    	  (l4) edge node {$\8|\9$} (an)
    	  (xn) edge node {$\8|\8$}(an)
          (e) edge [double,loop right] node {$\9|\9,\8|\8$} (e);
  \end{fsa}
\]
\caption{A Mealy automaton generating an infinite semigroup with a finitely generated infinite subsemigroup for which every orbit is finite.}\label{figure:SubsemigroupExample}
\end{figure}

We saw from the example of Figure \ref{figure:FiniteOrbits} that if we consider Mealy automata with an infinite number of states, it is possible for every orbit to be finite. However, in that example, the orbits were in fact bounded. Thus, it is natural to ask the following question: given a semigroup generated by a (not necessarily finite state) Mealy automaton, is it true that either there exists an infinite orbit or every orbit is bounded?

As can be seen in the following example, suggested to us by Laurent Bartholdi, it turns out that the answer to this question is negative.

\begin{example}
Let $A=\{0,1,2\}$ and $Q=\{e, a_{ij} \mid i\in \NN, 1\leq j \leq i^2\}$ be two sets, and let $\tau\colon Q\times A \rightarrow A\times Q$ be the map given by $\tau(e,x) = (x,e)$ for all $x\in A$ and
\begin{align*}
\tau(a_{ij},0) &= \begin{cases}
(1,a_{i(j-1)}) & \text{ if } j\equiv 1 \mod i\\
(0,e) & \text{ otherwise}
\end{cases} \\
\tau(a_{ij},1) &= \begin{cases}
(0,a_{i(j-1)}) & \text{ if } j\equiv 1 \mod i\\
(1,e) & \text{ otherwise}
\end{cases} \\
\tau(a_{ij},2) &= \begin{cases}
(2,e) & \text{ if } j\equiv 1 \mod i\\
(2,a_{i(j-1)}) & \text{ otherwise}
\end{cases}
\end{align*}
(where we set $a_{i0}=e$). We can then consider the automaton $\mathcal{A}=(Q,A,\tau)$. A part of its Moore diagram is represented in Figure \ref{figure:NonBoundedExample}.

\begin{figure}[h]
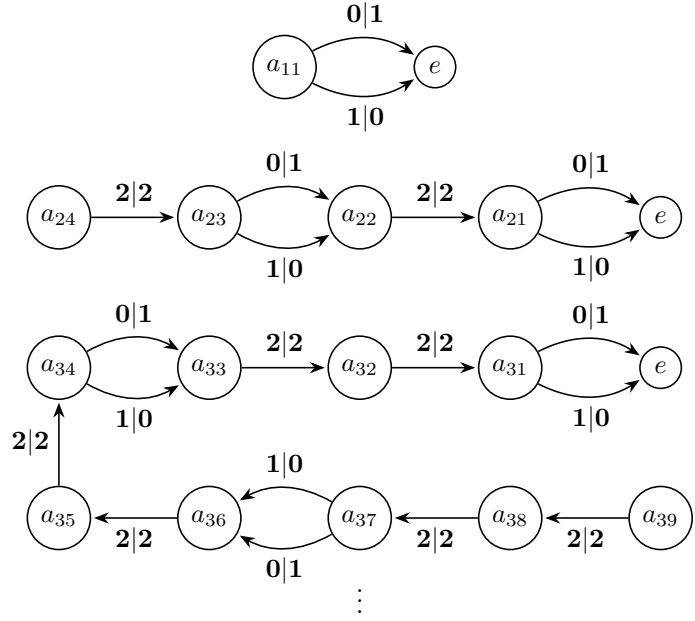

\centering

\[\begin{fsa}[baseline]
    \node[state] (e1) at (1,2) {$e$};
    \node[state] (a11) at (-1,2) {$a_{11}$};
    \node[state] (e2) at (4,0) {$e$};
    \node[state] (a21) at (2,0) {$a_{21}$};
    \node[state] (a22) at (0,0) {$a_{22}$};
    \node[state] (a23) at (-2,0) {$a_{23}$};
    \node[state] (a24) at (-4,0) {$a_{24}$};
    \node[state] (e3) at (4,-2) {$e$};
    \node[state] (a31) at (2,-2) {$a_{31}$};
    \node[state] (a32) at (0,-2) {$a_{32}$};
    \node[state] (a33) at (-2,-2) {$a_{33}$};
    \node[state] (a34) at (-4,-2) {$a_{34}$};
    \node[state] (a35) at (-4,-4) {$a_{35}$};
    \node[state] (a36) at (-2,-4) {$a_{36}$};
    \node[state] (a37) at (0,-4) {$a_{37}$};
    \node[state] (a38) at (2,-4) {$a_{38}$};
    \node[state] (a39) at (4,-4) {$a_{39}$};
    \node (d) at (0,-5) {$\vdots$};
    \path (a11) edge [bend right] node [below] {$\8|\9$} (e1)
    (a11) edge [bend left] node [above] {$\9|\8$} (e1)
    (a24) edge node {$\7|\7$} (a23)
    (a23) edge [bend right] node [below] {$\8|\9$} (a22)
    (a23) edge [bend left] node {$\9|\8$} (a22)
    (a22) edge node {$\7|\7$} (a21)
    (a21) edge [bend right] node [below] {$\8|\9$} (e2)
    (a21) edge [bend left] node [above] {$\9|\8$} (e2)
    (a39) edge node {$\7|\7$} (a38)
    (a38) edge node {$\7|\7$} (a37)
    (a37) edge [bend right] node [above] {$\8|\9$} (a36)
    (a37) edge [bend left] node {$\9|\8$} (a36)
    (a36) edge node {$\7|\7$} (a35)
    (a35) edge node {$\7|\7$} (a34)
    (a34) edge [bend right] node [below] {$\8|\9$} (a33)
    (a34) edge [bend left] node {$\9|\8$} (a33)
    (a33) edge node {$\7|\7$} (a32)
    (a32) edge node {$\7|\7$} (a31)
    (a31) edge [bend right] node [below] {$\8|\9$} (e3)
    (a31) edge [bend left] node [above] {$\9|\8$} (e3);
  \end{fsa}
\]
\caption{A part of the Moore diagram of a Mealy automaton where the orbits are finite but unbounded. To simplify the drawing, the state $e$ appears multiple times, and arrows that go from any state to the state $e$ by fixing a letter are not drawn.}\label{figure:NonBoundedExample}
\end{figure}

The semigroup $S$ generated by $\mathcal{A}$ is in fact a group, and since it is composed of finitary automorphisms of $A^\omega$, it is locally finite. We claim that all its orbits are finite but unbounded. To see this, let us first notice that, from the definition, for all $1\leq j \leq i^2$, all $n\in \NN$ such that $n\ne i$ and all $x,y\in \{0,1\}$, we have $a_{ij}@(x\concat 2^{n-1}\concat y) = e$. Therefore, for all $s\in S$, we have
\[s@(x\concat 2^{n-1}\concat y) \in \langle a_{n1}, a_{n2}, \dots, a_{nn^2} \rangle,\]
which is a finite subgroup of $S$.

Let us consider $\xi \in A^\omega$. There are two cases: either there exists some finite prefix $w$ of $\xi$ such that $w$ contains exactly two letters in the set $\{0,1\}$ or $\xi$ contains at most one letter in this set. In the second case, it is easy to see that the orbit of $\xi$ has size at most $2$, since elements of $S$ fix every occurence of the letter $2$ in any word.

Let us now assume that we are in the first case and let $n\in \NN$, $x,y\in \{0,1\}$ be such that $w$ contains the subword $x\concat 2^{n-1} \concat y$. Then, it follows from our previous remark that $s@w\in \langle a_{n1}, a_{n2}, \dots, a_{nn^2} \rangle$ for all $s\in S$. As $s\cdot \xi = (s\cdot w)\concat ((s@w)\cdot \xi')$, where $\xi=w\concat \xi'$, we conclude that the orbit of $\xi$ must be finite.

To see that the orbits of the action of $S$ are unbounded, it suffices to notice that $S$ acts transitively on the sets
\[V_{i} = \{a_1a_2\dots a_{i^2} \in A^{i^2} \mid a_k \in \{0,1\} \text{ if } k\equiv 1 \mod i, \quad a_k=2 \text{ otherwise}\}\]
for all $i\in \NN$.

\end{example}

\end{document}